\documentclass[11pt]{amsart}
\title{HAUSDORFFNESS OF GENERAL COMPACTIFICATIONS}
\author{{Ramkumar, S    and    Ganesa Moorthy, C }}

\address{Department of Mathematics\\ Alagappa University\\ Karaikudi - 630 003\\ India\\}
\email{ramkumarsolai@gmail.com and ganesamoorthyc@gmail.com}
\keywords{Semi-lattices of compactifications, Quotient topology, Hausdorff Partitions }
\subjclass[2000]{54D35}
\date{}

\begin{document}

\newtheorem{lem}{\sf Lemma}[section]
\newtheorem{defn}[lem]{\sf Definition}
\newtheorem{thm}[lem]{\sf Theorem}
\newtheorem{cor}[lem]{\sf Corollory}
\newtheorem{pro}[lem]{\sf Property}
\newtheorem{rem}[lem]{\sf Remark}

\begin{abstract}
Magill proved that the remainders of two locally compact Hausdorff spaces in their Stone-$\check{C}$ech compactifications are homeomorphic if and only if the lattices of their Hausdorff compactifications are lattice isomorphic. His construction for compactifications are explicitely discussed through the partitions of their Stone-$\check{C}$ech compactifications. Partitions in a Stone-$\check{C}$ech compactification which lead to Hausdorff compactifications are characterized in this article. Embeddings of certain upper semi-lattices of compactifications into lattices of compactifications are constructed.
\end{abstract}
\maketitle
{\section{Introduction}}

Let $X$ be a completely regular Hausdorff  space and $\alpha_1X$ and $\alpha_2X$ be two Hausdorff compactifications of $X$. These two compactifications may be compared by an order relation: $\alpha_1 X\geq\alpha_2X$ if and only if there is a continuous function $h_1:\alpha_1X\rightarrow \alpha_2X$ such that $h_1(x)=x$ for all $x\in X$. The collection $K(X)$ of all Hausdorff compactifications of a Tychonoff space $X$ forms a complete upper semi-lattice under the natural order defined above. It is known that for a Tychonoff space $X$, $K(X)$ is a lattice if and only if $X$ is locally compact (see: \cite[Theorem 4.3 (e)]{PGR}). Magill\cite{Mag} proved that the remainders $\beta X\setminus X$ and  $\beta Y\setminus Y$ of $X$ and $Y$ are homeomorphic if and only if $K(X)$ and $K(Y)$ are lattice isomorphic, where $X$ and $Y$ are locally compact spaces. Rayburn \cite{Ray} considered non locally compact points and obtained some extensions of Magill's results. These two articles are fundamental articles for studies on lattice structure on compactifications and topological structure of remainders. Magill furnished indirectly a construction for all Hausdorff compactifications of a given Tychonoff space. This construction is based on partitions in their Stone-$\check{C}$ech compactifications and this is explained in the first section. Every partition of a Stone-$\check{C}$ech compactification by compact subsets always leads to a compactification, which is the  corresponding quotient space. A characterization for partitions which lead to Hausdorff compactifications is discussed in the second section. If $Y$ is the collection of all locally compact points of a given Tychonoff space $X$ and if $Y$ is dense in $X$, then the upper semi-lattice $K(X)$ can be embedded into the lattice $K(Y)$. This is explained in the third section.\\
\section{Magill's Construction}

Let $K(X)$ be the collection of all Hausdorff compactifications of a Tychonoff space $X$ and $\beta X$ be its Stone-$\check{C}$ech compactification. For every $\alpha X\in K(X)$, there is a continuous map, called $\check{C}$ech map, $f_\alpha :\beta X \rightarrow \alpha X$ such that $f_\alpha (x)=x$ for all $x \in X$. Also $\{f_\alpha^{-1} (y) : y\in \alpha X\}$ forms a partiton in $\beta X$ , where each $f_\alpha^{-1} (y)$ is a compact subset of $\beta X$, when $y\in \alpha X$. Moreover $\{x\}$ is in this partition, for every $x\in X$. This is justified by the following lemma.\\
\begin{lem}
 
Let $\alpha_1 X$, $\alpha_2 X$ be two Hausdorff compactifications of $X$ such that $\alpha_1 X\geq \alpha_2 X$. Let $f:\alpha_1 X \rightarrow \alpha_2 X$ be the natural continuous onto mapping such that $f(x)=x$, for every $x\in X$. Then$f^{-1}(x)=\{x\}$, for every $x\in X$.
\end{lem}
\begin{proof}
On the contrary assume that, there is an element $y\in f^{-1}(x)$ such that  $y\neq x$, for some $x\in X$.
Then there are two disjoint open neighbourhoods $U_x,\ U_y$ of $x,\ y$ in $\alpha_1 X$, respectively. For $U_x\mathop{\cap}X$, find an open neighbourhood $V_x$ of $x$ in $\alpha_2 X$ such that $V_x\mathop{\cap}X=U_x\mathop{\cap}X$. Since $f(y)=x\in V_x$, find an open neighbourhood $W_y$ of $y$ in $\alpha_1 X$ such that $y\in W_y\subseteq U_y$ and $f(W_y)\subseteq V_x$.
Then $W_y\mathop{\cap}X=f(W_y\mathop{\cap}X)\subseteq V_x\mathop{\cap}X=U_x\mathop{\cap}X$. Thus $U_x\mathop{\cap}W_y\mathop{\cap}X\neq \phi$. This contradicts the fact that $U_x\mathop{\cap}W_y\mathop{\cap}X=\phi$. This proves the lemma.\end{proof}

On the other hand, consider a partition $\pi$ of $\beta X$ such that

\begin{itemize}
\item [(i)] Every member of $\pi$ is a compact subset of $\beta X$.
\item [(ii)]  $\{x\}\in \pi$, for every $x\in X$.
\end{itemize}

Now consider the quotient space $\beta X/\pi$ with the quotient topology induced by a quotient map $f: \beta X \rightarrow \beta X/ \pi$. Since the quotient map is continuous and $\beta X$ is compact, $f$ is surjective and $\beta X/ \pi$ is a compact space. Also $\beta X/ \pi$ is a compactification of $X$, because $X$ is dense in $\beta X/ \pi$. This is a construction of Magill \cite{Mag} for compactifications. But this compactification may not be Hausdorff unless $\pi$ is a Hausdorff partition of $\beta X$. That is, $\beta X/ \pi$ is made into a Hausdorff space under the quotient topology.
\begin{lem}

Let $X$ be a Hausdorff space and $\{K_i\}_{i\in I}$ be a collection of mutually disjoint non empty compact subsets of $X$ and it is locally finite in $X$. Then, for any fixed $K_m$, there is an open set $U$ such that $U$ contains $K_m$ and $U$ does not intersect any of the $K_j ,\ j\neq m$.
\end{lem}
\begin{proof} 
Let $x\in K_m$. Since $\{K_i\}_{i\in I}$ is locally finite, there is an open set $U$ of $x$  which intersects only finite number of $K_i$'s. Suppose $U$ intersects only $K_{i_1}, K_{i_2}\cdots K_{i_n}$ other than $K_m$. Then
$U_x=U\setminus\mathop{\cup}\limits_{k=i_1}^{i_n}K_k$ does not intersect none of the $K_i$'s other than $K_m$. Then $\{U_x: x\in K_m\}$ is an open cover for $K_m$ and their union is an open set which contains $K_m$ and does not intersect any of the $K_i,\ i\neq m$.\end{proof}
\begin{thm}
Let $X$ be a Tychonoff space and $\alpha X$ be any Hausdorff compactification of $X$. Let $\{K_i\}_{i\in I}$  be a collection of mutually disjoint non empty compact subsets of $\alpha X\setminus X$ such that it is locally finite in $\alpha X$. Then there is a Hausdorff compactification  $\gamma X=(\alpha X\setminus\mathop{\cup} \limits_{i\in I}K_i)\cup \{p_i:i\in I\}$ of $X$, where $p_i$ are distinct and $p_i\notin \alpha X$, and there is a continuous mapping $h:\alpha X\rightarrow \gamma X$ such that 
$h(x)=x$, for $x\notin \mathop{\cup}\limits_{i\in I}K_i$ and $h(x)=p_i$, for $x\in K_i$.
\end{thm}
\begin{proof}
Let $\gamma X=(\alpha X\setminus\mathop{\cup} \limits_{i\in I}K_i)\cup \{p_i:i\in I\}$ and $Y=(\alpha X\setminus\mathop{\cup} \limits_{i\in I}K_i)$ where $p_i$ are distinct, and $p_i\notin \alpha X$. Define a map $h:\alpha X\rightarrow \gamma X$ by $h(x)=x$ if $x\in Y$ and $h(x)=p_i$ if $x\in K_i$. Let $\gamma X$ have the quotient topology under the quotient map $h$. Since $\alpha X$ is compact, $\gamma X$ is compact. Let $U$ be an open set in $\gamma X$. Then $h^{-1}(U)$ is an open set in $\alpha X$ which intersects $X$ so that $h(h^{-1}(U))=U$ intersects $h(X)=X$. Hence $X$ is dense in $\gamma X$. To prove the Hausdorffness, we have to consider the following three cases for any $x,\  y\in \gamma X$ such that $x\neq y$.
\begin{itemize}
\item [(i)] $x\in \gamma X\setminus Y$ and $y\in \gamma X\setminus Y$.
\item [(ii)] $x\in Y$ and $y\in \gamma X\setminus Y$.
\item [(iii)] $x\in Y$ and $y\in Y$.

\end{itemize}
{\bf Case (i):}

Let $x\in \gamma X\setminus Y$ and $y\in \gamma X\setminus Y$. Then $x=p_i$ and $y=p_j$, $i\neq j$. Since $\alpha X$ is normal, we can find open sets $V$ and $W$ in $\alpha X$  such that $K_i\subseteq V$ and $K_j\subseteq W$ and $V\mathop{\cap}W= \phi$. Since $\{K_i:i\in I\}$ is locally finite in $\alpha X$, we can find open sets $V_1$ and $W_1$ such that  $K_i\subseteq V_1$ and $K_j\subseteq W_1$ and $V_1$ and $W_1$ does not intersect any of the $K_s$'s other than $K_i$ and $K_j$, respectively. Then $V\mathop{\cap}V_1$ and $W\mathop{\cap}W_1$ are open sets in $\alpha X$ such that $(V\mathop{\cap}V_1)\mathop{\cap}(W\mathop{\cap}W_1)=(V\mathop{\cap}W)\mathop{\cap}(V_1\mathop{\cap}W_1)=\phi$. Let $V_2=V\mathop{\cap}V_1$ and  $W_2=W\mathop{\cap}W_1$. Let  $V^{\star}=h(V_2)$ and  $W^{\star}=h(W_2)$. Since $h^{-1}(V^{\star})=V_2$ and $h^{-1}(W^{\star})=W_2$, $V^{\star}$ and $W^{\star}$ are disjoint open sets in $\gamma X$ such that $x\in V^{\star}$ and $y\in W^{\star}$. \\
{\bf Case (ii):}

Let $x\in Y$ and $y\in \gamma X\setminus Y$. Then $y=p_j$ for some $j\in I$. Since $\alpha X$ is normal, we can find open sets $V$ and $W$ in $\alpha X$ such that $x\in V$, $K_j\subseteq W$ and $V\mathop{\cap}W= \phi$. Since $\{K_i\}_{i\in I}$ is locally finite, there is an open set $U$ of $x$  which intersects only finite number of $K_k$'s. Suppose $U$ intersects $K_{i_1}, K_{i_2}\cdots K_{i_n}$. Then
$U_1=U\setminus\mathop{\cup}\limits_{k=i_1}^{i_n}K_k$ does not intersect none of the $K_k$'s. Similarly, find an open set $U_2$ containing $K_j$, but not containing other $K_k$'s. Let $V_1=h(U_1\mathop{\cap}V)$ and $W_1=h(U_2\mathop{\cap}W)$. Since $h^{-1}(V_1)=U_1\mathop{\cap}V$ and $h^{-1}(W_1)=U_2\mathop{\cap}W$, $V_1$ and $W_1$ are open sets in $\gamma X$ containing $x$ and $p_j$, respectively such that their intersection is empty. 
\newpage
\noindent{\bf Case (iii):}

Let $x\in Y$ and $y\in Y$. Since $\{K_i\}_{i\in I}$ is locally finite, there exist disjoint open sets $U$ and $V$ in $\alpha X$ containing $x$ and $y$, respectively such that they intersect only finite number of $K_i$'s. Suppose $U$ intersects $K_{i_1}, K_{i_2}\cdots K_{i_n}$ and $V$ intersects $K_{j_1}, K_{j_2}\cdots K_{j_m}$. Since $\alpha X$ is Hausdorff, there exist disjoint open sets $U_1$ and $V_1$ such that $x\in U_1$ and $y\in V_1$. Let  $U_2=(U\setminus\mathop{\cup}\limits_{k=i_1}^{i_n}K_k)\mathop{\cap}U_1$ and $V_2=(V\setminus\mathop{\cup}\limits_{k=j_1}^{j_m}K_k)\mathop{\cap}V_1$. Since $h(U_2)=U_2$ and $h(V_2)=V_2$, $U_2$ and $V_2$ are disjoint open sets in $\gamma X$ containing $x$ and $y$, respectively. Hence $\gamma X$ is a Hausdorff compactification of $X$. \end{proof}
\begin{rem}
This theorem 2.3 generalizes the lemma 2 in \cite{Mag}. If $K_i$ are selected in 
$\alpha X= X\mathop{\cup}(\alpha X\setminus X)$, then theorem 2.3 is true except the fact that $\gamma X$ is just a compact Hausdorff space; but not a compactification of $X$.
\end{rem}
{\section{Hausdorff partitions}}

Hausdorff partitons lead to Hausdorff compactifications. A characterization for Hausdorff partitions is obtained in this section.

Let $X$ be a Tychonoff space with its Stone-$\check{C}$ech compactification $\beta X$. Let $\pi$ be a partition of $\beta X$ such that 
\begin{itemize}
\item [(i)] Every member of $\pi$ is a compact subset of $\beta X$. 
\item [(ii)] $\{x\}\in \pi$, for every $x\in X$.
\end{itemize}
Then we have the following theorem.
\begin{thm}
Let $X$ be a Tychonoff space and $\pi$ be a partition of its Stone-$\check{C}$ech compactification $\beta X$. Then $\beta X/\pi$ is a Hausdorff compactification of $X$ under the quotient topology if and only if for every $A\in \pi$ and for every open subset $U$ of $\beta X$ such that $A\subseteq U$, there is an open subset $V$ of $\beta X$ such that $(i)\ A\subseteq V\subseteq U$ $(ii)\ V$ is a union of members of $\pi$.
\end{thm}
\begin{proof}
Let $f: \beta X \rightarrow \beta X/ \pi$ be the quotient map and $\beta X/ \pi$ be endowed with the quotient topology.

Suppose  $\beta X/ \pi$  is Hausdorff. Let $A\in \pi$ and $U$ be an open subset of $\beta X$ such that $A\subseteq U$. Then $\beta X\setminus U$ is closed and hence is a compact subset of $\beta X$. Since $f$ is continuous, $f(\beta X\setminus U)$ is compact. Since $\beta X/ \pi$  is Hausdorff, $f(\beta X\setminus U)$ is closed in $\beta X/\pi$. Moreover $f(A)$ is a singleton subset of  $\beta X/\pi$ and it is contained in the open set  $(\beta X/\pi)\setminus f(\beta X\setminus U$). Choose $V=f^{-1}((\beta X/\pi)\setminus (f(\beta X\setminus U)))$. Then $V$ is an open subset of $\beta X$ such that $(i)$ and $(ii)$ are true.

Conversely, assume that for every $A\in \pi$ and for every open subset $U$ of $\beta X$ such that $A\subseteq U$, there is an open subset $V$ of $\beta X$ such that $(i)$ and $(ii)$ are true. Let us fix two distinct points $u_1$ and $u_2$ in $\beta X/\pi$. Let $A_1=f^{-1}(u_1)$ and $A_2=f^{-1}(u_2)$. Since f is continuous, $A_1$ and $A_2$ are closed subsets of $\beta X$. Since $\beta X$ is normal, there are disjoint open subsets $U_1$ and $U_2$ in $\beta X$ such that $A_1\subseteq U_1$ and $A_2\subseteq U_2$. By assumption, there are disjoint open subsets $V_1$ and $V_2$ of $\beta X$ such that $A_1\subseteq V_1\subseteq U_1$ and $A_2\subseteq V_2 \subseteq U_2$ and $V_1$ and $V_2$ are unions of members of $\pi$. Now $f(V_1)$ and $f(V_2)$ are disjoint open subsets of $\beta X/\pi$ such that $u_1\in f(V_1)$ and $u_2\in f(V_2)$. This proves that 
$\beta X/\pi$ is Hausdorff. This completes the proof of the theorem.\end{proof}

          In the previous theorem $\beta X$ may be replaced by any other compactification $\alpha X$ of $X$.
{\section{Embedding into Lattices}}

A point in a Tychonoff space $X$ is locally compact in $X$ if it has a compact neighbourhood in $X$. Let $Y$ be the collection of all locally compact points of a Tychonoff space $X$. Suppose $Y$ is dense in $X$. 
( For example, let $X$ be the closed unit disc without some points on the unit circle and $Y$ be an open unit disc in the Euclidean plane). Then $Y$ is dense in $\beta X$ and $Y$ is locally compact. So $Y$ is open in $\beta X$ (see:\cite[Theorem 4.3]{PGR}). The collection $K(X)$ of all Hausdorff compactificaions of $X$ is a complete upper semi-lattice. Since $Y$ is locally compact, $K(Y)$ is a complete lattice. Now $K(X)$ is considered as a subset of $K(Y)$, because every Hausdorff compactification of $X$ is a Hausdorff compactification of $Y$. This identification is an order preserving map. The construction explained in section 1  reveals that this order preserving map also preserves join. Note that the join of two Hausdorff compactifications given by two partitions $\pi_1$, $\pi_2$ of $\beta X$ is given by the partition $\{ A\cap B: A\in \pi_1, B\in \pi_2\}\setminus \{\phi\}$. So we have the following theorem.\\
\begin{thm}
If $Y$ is the set of all locally compact points of a Tychonoff space $X$ and if $Y$ is dense in $X$, then the complete upper semi-lattice $K(X)$ can be embedded into the lattice $K(Y)$ by an order preserving map which also preserves join.
\end{thm}

\end{document}